\documentclass{birkjour}
 \newtheorem{thm}{Theorem}[section]
 \newtheorem{cor}[thm]{Corollary}
 \newtheorem{lem}[thm]{Lemma}
 \newtheorem{prop}[thm]{Proposition}
 \theoremstyle{definition}
 \newtheorem{defn}[thm]{Definition}
 
 \theoremstyle{remark}
 
 \numberwithin{equation}{section}
\usepackage{color}
\begin{document}

%
%
%
%
%
%
%
%
%

\title[Dual and multiplier of $K$-fusion frames]
 {Dual and multiplier of $K$-fusion frames}

\author[M. Shamsabadi]{Mitra Shamsabadi}
\address{Department of Mathematics and Computer Sciences, Hakim Sabzevari University, Sabzevar, Iran.}
\email{mi.shamsabadi@hsu.ac.ir}

\author[A. Arefijamaal]{Ali Akbar Arefijamaal}
\address{Department of Mathematics and Computer Sciences, Hakim Sabzevari University, Sabzevar, Iran.}
\email{arefijamaal@hsu.ac.ir;arefijamaal@gmail.com}

\author[Gh. Sadeghi]{Ghadir Sadeghi}
\address{Department of Mathematics and Computer Sciences, Hakim Sabzevari University, Sabzevar, Iran.}
\email{g.sadeghi@hsu.ac.ir}


\begin{abstract}
 In this paper, we introduce the concept of $K$-fusion frames and  propose the duality for such frames. The relation between the local frames of $K$-fusion frames with their dual is studied. The elements from the range of a bounded linear operator $K$ can be reconstructed by $K$-frames. Also, we establish $K$-fusion frame multipliers and investigate reconstruction of the range of $K$ by them.
\end{abstract}

\maketitle

\section{Introduction, notation and motivation}
The theory of frames  plays an important role in wavelet theory as well as (time-frequency)  analysis for functions in $L^{2}(\mathbb{R}^{d})$ \cite{ole,Feich}.
The traditional applications of frames are signal processing, image processing \cite{Can04}, sampling theory and communication \cite{H}, moreover,
recently the use of frames also in numerical analysis for the solution of
operator equation by adaptive schemes is investigated \cite{Ba10}.
Also, frame  multipliers have so applications in psychoacoustical modeling and denoising \cite{lab,maj}. 

For two sequences $\Phi:=\{\phi_{i}\}_{i\in I}$ and
 $\Psi:=\{\psi_{i}\}_{i\in I}$ in a Hilbert space $\mathcal{H}$
 and a sequence $m=\{m_{i}\}_{i\in I}$ of complex scalars, the operator
  $\mathbb{M}_{m,\Phi,\Psi}:\mathcal {H}\rightarrow \mathcal{H}$ given by
\begin{eqnarray}\label{mu}
\mathbb{M}_{m,\Phi,\Psi}f=\sum_{i\in I} m_{i}\langle f,\psi_{i}\rangle\varphi_{i}, \qquad  (f\in\mathcal{H}),
\end{eqnarray}
is called a \textit{multiplier}. The sequence $m$ is called
\textit{symbol}. If $\Phi$ and $\Psi$ are Bessel sequences for
 $\mathcal{H}$ and $m\in \ell^{\infty}$, then $\mathbb{M}_{m,\Phi,\Psi}$
  is well defined  and $\|\mathbb{M}_{m,\Phi,\Psi}\|\leq \sqrt{B_{\Phi}B
  _{\Psi}}\|m\|_{\infty}$, where $B_{\Phi}$ and $B_{\Psi}$ are Bessel bounds of $\Phi$ and $\Psi$, respectively \cite{Basic}. The
  invertibility of multipliers, which plays a key role in the
   topic, is discussed in \cite{Basic,rep,inv}.

$K$-frames which recently introduced by G\v{a}vru\c{t}a are a generalization of frames, in the meaning that the lower frame bound only holds for that admits to reconstruct from the range of a linear and bounded operator $K$ in a Hilbert space.

In this section, we briefly recall the basic concepts of $K$-frames  and their properties \cite{arab, Ga, Xiao}.
\begin{defn}
Let $K$ be a bounded and linear operator on a separable Hilbert space $\mathcal{H}$. A sequence $F:=\lbrace f_{i}\rbrace_{i \in I} \subseteq \mathcal{H}$ is called a $K$-frame for $\mathcal{H}$, if there exist constants $A, B > 0$ such that
\begin{eqnarray}\label{1}
A \Vert K^{*}f\Vert^{2} \leq \sum_{i\in I} \vert \langle f,f_{i}\rangle\vert^{2} \leq B \Vert f\Vert^{2}, \quad (f\in \mathcal{H}).
\end{eqnarray}
\end{defn}
Clearly if $K=I_{\mathcal{H}}$, then $F$ is an ordinary frame. The constants $A$ and $B$ in $(\ref{1})$ are called  lower and  upper bounds of $F$, respectively.  If $A=B=1$ we call $F,$ a Parseval $K$-frame. Obviously, every $K$-frame is a Bessel sequence, hence similar to ordinary frames the synthesis operator can be defined as $T_{F}: l^{2}\rightarrow \mathcal{H}$; $T_{F}(\{ c_{i}\}_{i\in I}) = \sum_{i\in I} c_{i}f_{i}$. It is a bounded operator and its adjoint, which is called the analysis operator, is given by $T_{F}^{*}(f)= \{ \langle f,f_{i}\rangle\}_{i\in I}$. Finally, the frame operator is defined by $S_{F}: \mathcal{H} \rightarrow \mathcal{H}$; $S_{F}f = T_{F}T_{F}^{*}f = \sum_{i\in I}\langle f,f_{i}\rangle f_{i}$. Some properties of ordinary frames are not hold for $K$-frames, for example, the frame operator of a $K$-frame is not invertible and duality is not interchangeable, in general \cite{Xiao}. If $K$ has close range then $S_{F}$ from $R(K)$ onto $S_{F}(R(K))$ is an invertible operator \cite{Xiao} and
\begin{eqnarray}\label{bound S}
B^{-1} \| f\| \leq \|S_{F}^{-1}f\| \leq A^{-1}\|K^{\dag}\|^{2}\| f\|, \quad (f\in S_{F}(R(K))),
\end{eqnarray}
where $K^{\dag}$ is the \textit{pseudo-inverse} of $K$. For further information in $K$-frames refer to \cite{arab,Xiao}.
\begin{defn}\cite{arab}
Let $\{f_{i}\}_{i\in I}$ be a $K$-frame. A Bessel  sequence $\{g_{i}\}_{i\in I}$ is called a \textit{$K$-dual} of $\{f_{i}\}_{i\in I}$ if
\begin{eqnarray}\label{dual1}
Kf = \sum_{i\in I} \langle f,g_{i}\rangle \pi_{R(K)}f_{i}, \quad (f\in \mathcal{H}).
\end{eqnarray}

The $K$-dual $\{K^{*}S_{F}^{-1}\pi_{S_{F}R(K)}f_{i}\}_{i\in I}$ of $F=\{f_{i}\}_{i\in I}$ which is called the \textit{canonical dual}, is denoted by $\{\widetilde{f_{i}}\}_{i\in I}$.
\end{defn}

In the present paper,  the reconstruction  elements from the range of $K$ by a $K$-fusion frame,  where $K$  is a closed range and bounded linear operator on $\mathcal{H}$, is investigated. We also introduce $K$-fusion frame multipliers and discuss their invertibility.

Throughout this paper, we suppose that $\mathcal{H}$ is a separable Hilbert space, $I$ a countable index set and $I_{\mathcal{H}}$  the identity operator on $\mathcal{H}$. We denote by $B(\mathcal{H}_{1},\mathcal{H}_{2})$ the collection of all bounded linear operators between  Hilbert spaces $\mathcal{H}_{1}$ and $\mathcal{H}_{2}$, and abbreviate $B(\mathcal{H},\mathcal{H})$ by $B(\mathcal{H})$. Also we denote the range of $K\in B(\mathcal{H})$ by $R(K)$ and $\pi_{V}$ denotes the orthogonal projection of $\mathcal{H}$ onto a closed subspace $V \subseteq \mathcal{H}$.

We end this section with a vital proposition which frequently will be used.

%
\begin{prop} \cite{Gav07}\label{k}
Let $L_{1},L_{2}\in B(\mathcal{H})$ be two bounded operators. The following statements  are equivalent:
\begin{enumerate}
\item \label{Re1} $R(L_{1})\subset R(L_{2})$.
\item \label{Re2} $L_{1}L_{1}^{*}\leq \lambda^{2}L_{2}L_{2}^{*}$
for some $\lambda\geq 0$.
\item \label{Re3} There exists a bounded operator $X\in B
(\mathcal{H})$ so that $L_{1}=L_{2}X$.
\end{enumerate}
\end{prop}


\section{$K$-fusion frames}
In this section, we present $K$-fusion frames and discuss their properties. Moreover, we focus on the duality, which  is different from the ordinary frames, and obtain some characterizations  of dual $K$-fusion frames.
\begin{defn}
Let $\{W_{i}\}_{i\in I}$ be a family of closed subspaces of $\mathcal{H}$ and $\{\omega _{i}\}_{i\in I}$
 a family of weights, i.e. $\omega_{i}>0, i\in I$. The sequence  $\{(W_{i},\omega_{i}
)\}_{i\in I}$
is called a \textit{$K$-fusion frame} for $\mathcal{H}$ if there exist constants $0<A\leq B<\infty$ such that
\begin{eqnarray}\label{fusion}
A\|K^{*}f\|^{2}\leq\sum_{i\in I}\omega _{i}^{2}\|\pi_{W_{i}}f\|^{2}\leq B\|f\|^{2}, \qquad (f\in \mathcal{H}).
\end{eqnarray}
The constants $A$ and $B$ are called the  \textit{$K$-fusion frame bounds}. Obviously, every $K$-fusion frame is a Bessel fusion  sequence. If $A=B=1$ we call it a Parseval $K$-fusion frame. Similar to fusion frames, for a Bessel fusion sequence $\{(W_{i},\omega_{i})\}_{i\in I}$ we define the \textit{synthesis operator}
 $T_{W}:\left(\sum_{i\in I}\bigoplus W_{i}\right)_{\ell^{2}}\rightarrow \mathcal{H}$
    by
    \begin{eqnarray*}
T_{W}(\{f_{i}\}_{i\in I})=\sum_{i\in I}\omega_{i}f_{i}.
\end{eqnarray*}
Its adjoint operator $T^{*}_{W}: \mathcal{H}\rightarrow \left(\sum_{i\in I}\bigoplus W_{i}\right)_{\ell^{2}}$, which is called the \textit{analysis operator}, is obtained
 by
$
T^{*}_{W}f=\{\omega_{i}\pi _{W_{i}}f\}_{i\in I},
$ where
\begin{eqnarray*}
\left(\sum_{i\in I}\bigoplus W_{i}\right)_{\ell^{2}}=\left\{\{f_{i}\}_{i\in I}: f_{i}\in W_{i}, \sum_{i\in I}\|f_{i}\|^{2}<\infty \right\}
 \end{eqnarray*}
is a Hilbert space. Also the $\textit{frame operator}$  of $\{W_{i}\}_{i\in I}$ on $\mathcal{H}$, denoted by $S_{W}$, is given by
 \begin{eqnarray*}
S_{W}f=T_{W}T_{W}^{*}=\sum_{i\in I}\omega_{i}^{2}\pi_{W_{i}}f.
 \end{eqnarray*}
It is not difficult  to see that the frame operator of a $K$-fusion frame is not invertible on $\mathcal{H}$, in general. However, $S_{W}:R(K)\rightarrow S_{W}R(K)$ is invertible and
\begin{eqnarray}\label{karan}
B^{-1}\|f\|\leq \|S_{W}^{-1}f\|\leq A^{-1}\|K^{\dag}\|^{2}\|f\|,\qquad
(f\in S_{W}(R(K))),
\end{eqnarray}
where $K^{\dag}$ is the pseudo-inverse of $K$.
\end{defn}

Now, we can reconstruct $R(K)$ by $K$-fusion frame elements.
\begin{eqnarray*}
Kf=S_{W}^{*}(S_{W}^{-1})^{*}Kf=\sum_{i\in I}\omega_{i}^{2}\pi_{R(K)}\pi_{W_{i}}(S_{W}^{-1})^{*}Kf, \qquad(f\in \mathcal{H}).
\end{eqnarray*}

\begin{defn}
  Let  $W=\{(W_{i},\omega_{i})\}_{i\in I}$ be a $K$-fusion frame. A Bessel fusion sequence $V=\{(V_{i},\upsilon_{i})\}_{i\in I}$ is called a $K$-dual for $W$ if
  \begin{eqnarray}\label{dual}
  Kf=\sum_{i\in I}\omega_{i}\upsilon_{i}\pi_{R(K)}\pi_{W_{i}}
  (S_{W}^{-1})^{*}K\pi_{V_{i}}f,\qquad (f\in \mathcal{H}).
  \end{eqnarray}
\end{defn}
 It is easy to see that a Bessel fusion sequence  $V=\{(V_{i},\upsilon_{i})\}_{i\in I}$ is a $K$-dual of $K$-fusion frame $W=\{(W_{i},\omega_{i})\}_{i\in I}$
if and only if $\pi_{R(K)}T_{W}\psi_{wv}T_{V}^{*}=K$, where the bounded operator $\psi_{wv}:\left(\sum_{i\in I}\bigoplus V_{i}\right)_{\ell^{2}}\rightarrow \left(\sum_{i\in I}\bigoplus W_{i}\right)_{\ell^{2}}$  is given by
\begin{eqnarray*}\label{si}
\psi_{wv}\{g_{i}\}_{i\in I}=\{\pi_{W_{i}}(S_{W}^{-1})^{*}Kg_{i}\}_{i\in I}.
\end{eqnarray*}
One can see that every Bessel fusion sequence  $W=\{(W_{i},\omega_{i})\}_{i\in I}$ in $\mathcal{H}$  can be considered as a $K$-fusion frame for $\mathcal{H}$, if and only if  $R(K)\subset R(T_{W})$ by Proposition \ref{k}.

Every $K$-dual of a $K$-fusion frame is a $K^{*}$-fusion frame. Indeed,
let $V=\{(V_{i},\upsilon_{i})\}_{i\in I}$ be a $K$-dual  of a $K$-fusion frame  $W=\{(W_{i},\omega_{i})\}_{i\in I}$. Then
\begin{eqnarray*}
\|Kf\|^{2}&=&
\left|\left\langle \sum_{i\in I}\omega_{i}\upsilon_{i}\pi_{R(K)}\pi_{W_{i}}
  (S_{W}^{-1})^{*}K\pi_{V_{i}}f,Kf\right\rangle\right|\\
  &\leq &\sum_{i\in I}\omega_{i}\upsilon_{i}\left|\left\langle (S_{W}^{-1})^{*}K\pi_{V_{i}}f,\pi_{R(K)}\pi_{W_{i}}Kf\right
\rangle\right|\\
  &\leq &\sum_{i\in I}\omega_{i}\upsilon_{i}\left\|(S_{W}^{-1})^{*}K\pi_{V_{i}}f\right\|
  \left\|\pi_{W_{i}}Kf\right\|\\
  &\leq& \left(\sum_{i\in I}\upsilon_{i}^{2}\left\|(S_{W}^{-1})^{*}K\pi_{V_{i}}f\right\|^{2}\right)^
  \frac{1}{2}\left(\sum_{i\in I}\omega_{i}^{2}\left\|\pi_{W_{i}}Kf\right\|^{2}\right)^
  \frac{1}{2}\\
  &\leq&\|(S_{W}^{-1})^{*}K\|\sqrt{B_{W}}\|Kf\|\left(\sum_{i\in I}\upsilon_{i}^{2}\left\|\pi_{V_{i}}f\right\|^{2}\right)^
  \frac{1}{2},
\end{eqnarray*}
where $B_{W}$ is an upper bound of $W$ and  $f\in \mathcal{H}$.

Obviously, (\ref{fusion}) and (\ref{dual}) define a fusion frame and an ordinary dual fusion frame when $K$ is the identity operator on $\mathcal{H}$.


In the sequel, we need a key lemma for some characterizations of $K$-dual fusion frames.

\begin{lem}\label{v}
\cite{Ga} Let $V$ be a closed subspace of $\mathcal{H}$ and  $T\in B(\mathcal{H})$. Then
\begin{eqnarray*}
\pi_{V}T^{*}=\pi_{V}T^{*}\pi_{\overline{TV}}.
\end{eqnarray*}
\end{lem}

\begin{prop}
Let $W=\{(W_{i},\omega_{i})\}_{i\in I}$ be a Bessel fusion sequence such that $W_{i}\subseteq S_{W}(R(K))$, for all $i\in I$. Then $\widetilde{W}:=\{(K^{*}S_{W}^{-1}\pi_{S_{W}(R(K))}W_{i},\omega_{i})\}_{i\in I}$ is a $K$-dual of $W$.
\end{prop}
\begin{proof}
Applying Lemma \ref{v} yields
\begin{eqnarray*}
Kf&=&S_{W}^{*}\left(S_{W}^{-1}\right)^{*}Kf\\
&=&\sum_{i\in I}\omega_{i}^{2}\pi_{R(K)}\pi_{W_{i}}\left(S_{W}^{-1}\right)^{*}Kf\\
&=&\sum_{i\in I}\omega_{i}^{2}\pi_{R(K)}\pi_{W_{i}}\left(S_{W}^{-1}\right)^{*}K\pi
_{K^{*}S_{W}^{-1}\pi_{S_{W}(R(K))}W_{i}}f,
\end{eqnarray*}
for all $f\in \mathcal{H}$. So $\widetilde{W}$ satisfies in (\ref{dual}). It is enough to show that $\widetilde{W}$ is a Bessel fusion sequence in $\mathcal{H}$. Note that $S_{W}^{*}(S_{W}^{-1})^{*}=I_{\mathcal{H}}$ on $R(K)$ and $K^{\dag}K$ is an orthogonal projection on $R(K^{\dag})$. Using Lemma \ref{v}
for all $f\in R(K^{\dag})$ we obtain
\begin{eqnarray*}
\sum_{i\in I}\omega_{i}^{2}\left\|\pi_{\widetilde{W_{i}}}f\right\|^{2}&=&
\sum_{i\in I}\omega_{i}^{2}\left\|\pi_{\widetilde{W_{i}}}K^{\dag}S_{W}^{*}(S_{W}^{-1})^
{*}Kf\right\|^{2}\\
&=&\sum_{i\in I}\omega_{i}^{2}\left\|\pi_{\widetilde{W_{i}}}K^{\dag}S_{W}^{*}
\pi_{S_{W}(K^{\dag})^{*
}\widetilde{W_{i}}}(S_{W}^{-1})^
{*}Kf\right\|^{2}\\
&=& \sum_{i\in I}\omega_{i}^{2}\left\|\pi_{\widetilde{W_{i}}}K^{\dag}S_{W}^{*}
\pi_{W_{i}}(S_{W}^{-1})^
{*}Kf\right\|^{2}\\
&\leq& \left\|K^{\dag}\right\|^{2}\|S_{W}\|^{2}\sum_{i\in I}\omega_{i}^{2}\left\|
\pi_{W_{i}}(S_{W}^{-1})^
{*}Kf\right\|^{2}
\leq B\|f\|^{2},
\end{eqnarray*}
for some $B>0$. Now, if $f\in \mathcal{H}$ then there exist $f_{1}\in R(K^{\dag})$
    and $f_{2}\in R(K^{\dag})^{\bot}$ such that $f=f_{1}+f_{2}$. On the other hand
    \begin{eqnarray*}
 f_{2}\in \left(R(K^{\dag})\right)^
 {\perp}=\left(R(K^{*})\right)^{\perp}\subseteq \left(\widetilde{W_{i}}\right)^{\perp},
 \end{eqnarray*}
 for all $i\in I$.
  Hence
 \begin{eqnarray*}
 \sum_{i\in I}\omega_{i}^{2}\left\|\pi_{\widetilde{W_{i}}}f\right\|^{2}
 &=& \sum_{i\in I}\omega_{i}^{2}\left\|\pi_{\widetilde{W_{i}}}(f_{1}+f_{2})\right\|^{2}\\
 &=&
 \sum_{i\in I}\omega_{i}^{2}\left\|\pi_{\widetilde{W_{i}}}f_{1}\right\|^{2}\\
 &\leq& B\|f_{1}\|^{2}
\leq B\|f\|^{2}.
 \end{eqnarray*}
 for all $f\in \mathcal{H}$.
\end{proof}
The $K$-dual fusion frame $\widetilde{W}$ of a $K$-fusion frame $W$ is called the \textit{canonical $K$-dual} of $W$.


The following important theorem can be proved similar to Theorem 3.2 of \cite{ca}.
\begin{thm}\label{local}
Let $\{W_{i}\}_{i\in I}$ be a sequence of closed subspaces of $\mathcal{H}$, $\omega_{i}>0$, for each $i\in I$ and  $\{f_{ij}\}_{j\in J_{i}}$ be a frame
 for
$W_{i}$ with the frame bounds $A_{i}$ and $B_{i}$. Also assume that
\begin{eqnarray}\label{shart}
0<\inf_{i\in I}A_{i}\leq \sup_{i\in I}B_{i}< \infty.
\end{eqnarray}
Then $\{\omega_{i}f_{ij}\}_{i\in I,j\in J_{i}}$ is a $K$-frame for $\mathcal{H}$
if and only if $\{(W_{i},\omega_{i})\}_{i\in I}$
is  a $K$-fusion frame for $\mathcal{H}$.
\end{thm}
In the next, we investigate  the relation between the  local frames satisfying (\ref{shart})
of  $K$-fusion frames with their duals.

\begin{thm}
Let $W=\{(W_{i},\omega_{i})\}_{i\in I}$ be a $K$-fusion frame for $\mathcal{H}$
with the local frames $\{f_{ij}\}_{j\in J_{i}}$ for each $i\in I$. If $\{\widetilde{f}_{ij}\}_{
j\in J_{i}}$ is the canonical dual frame of $\{f_{ij}\}_{j\in J_{i}}$, then
\begin{enumerate}
\item\label{Re1}
$\{
K^{*}\omega_{i}f_{ij}\}_{i\in I
, j\in J_{i}}$ is a $K$-dual
of $\{S_{W}^{-1}\pi_{S_{W}(R(K))}\omega_{i}\widetilde{f}_{ij}\}_{i\in I, j\in J_{i}}$.


\item\label{Re3} $\{
 K^{*}S_{W}^{-1}\pi_{S_{W}(R(K))}\omega_{i}\widetilde{f}_{ij}\}_{i\in I,j\in J_{i}}$
 is a $K$-dual  for  $\{\omega_{i}f_{ij}\}_{i\in I,j\in J_{i}}$.

\end{enumerate}
\end{thm}
\begin{proof}
\ref{Re1}.
By using the fact that $\{\widetilde{f}_{ij}\}_{j\in J_{i}}$ is  the canonical dual of $\{f_{ij}\}_{j\in J_{i}}$, we obtain
\begin{eqnarray*}
\pi_{W_{i}}f=\sum_{j\in J_{i}}\left\langle f,f_{ij}\right\rangle \widetilde{f}_{ij},\qquad(i\in I, f\in \mathcal{H}).
\end{eqnarray*}
 Hence,
\begin{eqnarray*}
Kf&=&S_{W}^{-1}S_{W}Kf\\
&=&S_{W}^{-1}\sum_{i\in I}\omega_{i}^{2}\pi_{W_{i}}Kf\\
&=&S_{W}^{-1}\sum_{i\in I}\sum_{j\in J_{i}}\omega_{i}^{2}\left\langle Kf,f_{ij}
\right\rangle \widetilde{f}_{ij}\\
&=&\sum_{i\in I}\sum_{j\in J_{i}}\left\langle f, K^{*}\omega_{i}f_{ij}
\right\rangle \pi_{R(K)}S_{W}^{-1}\pi_{S_{W}(R(K))}\omega_{i}\widetilde{f}_{ij}.
\end{eqnarray*}
It remains to show that $\{K^{*}\omega_{i}f_{ij}\}_{i\in I,j\in J_{i}}$ and
 $\{S_{W}^{-1}\pi_{S_{W}(R(K))}\omega_{i}\widetilde{f}_{ij}\}_{i\in I,j\in J_{i}}$ are Bessel
 sequences. It is known that  $\{\omega_{i}f_{ij}\}_{i\in I,j\in J_{i}}$ and $\{\omega_{i}\widetilde{f}_{ij}\}_{i\in I,j\in J_{i}}$ are two $K$-frames for $\mathcal{H}$ by Theorem \ref{local}, since   $\{f_{ij}\}_{j\in J_{i}}$ and $\{\widetilde{f}_{ij}\}_
 {j\in J_{i}}$
are  the local frames  of $W_{i}$. Hence
 \begin{eqnarray*}
 \sum_{i\in I}\sum_{j\in J_{i}}\left|\left\langle f,K^{*}\omega_{i}f_{ij}\right
 \rangle\right|^{2}
 &=&\sum_{i\in I}\sum_{j\in J_{i}}
 \left|\left\langle Kf,\omega_{i}f_{ij}\right
 \rangle\right|^{2}\\
 &\leq& B\|K\|^{2}\|f\|^{2},
 \end{eqnarray*}
 for all $f\in \mathcal{H}$. Moreover,
 \begin{eqnarray*}
 \sum_{i\in I}\sum_{j\in J_{i}}\left|\left\langle f,S_{W}^{-1}\pi_{S_{W}(R(K))}\omega_{i}\widetilde{f}_{ij}
 \right
 \rangle\right|^{2}
 &=&\sum_{i\in I}\sum_{j\in J_{i}}
 \left|\left\langle (S_{W}^{-1})^{*}\pi_{R(K)}f,\omega_{i}\widetilde{f}_{ij}\right
 \rangle\right|^{2}\\
 &\leq& D\|S_{W}^{-1}\|^{2}\|f\|^{2},
 \end{eqnarray*}
 where $B$ and $D$ are upper bounds $\{\omega_{i}f_{ij}\}_{i\in I,j\in J_{i}}$ and $\{\omega_{i}\widetilde{f}_{ij}\}_{i\in I,j\in J_{i}}$, respectively.

\ref{Re3}.
The sequence $\{
 K^{*}S_{W}^{-1}\pi_{S_{W}(R(K))}\omega_{i}\widetilde{f}_{ij}\}_{i\in I,j\in J_{i}}$ is a Bessel sequence in $\mathcal{H}$  since  $\{\omega_{i}\widetilde{f}_{ij}\}_{i\in I,j\in J_{i}}$ is a $K$-frame for $\mathcal{H}$.
Hence
\begin{eqnarray*}
\pi_{R(K)}\pi_{W_{i}}f=\sum_{j\in J_{i}}\left\langle f,\widetilde{f}_{ij}
\right
\rangle\pi_{R(K)} f_{ij}.
\end{eqnarray*}
This follows that
\begin{eqnarray*}
Kf
&=&S_{W}^{*}\left(S_{W}^{-1}\right)^{*}Kf\\
&=&\pi_{R(K)}S_W\left(S_{W}^{-1}\right)^{*}Kf\\
&=&\sum_{i\in I}\pi_{R(K)}\pi_{W_{i}}\left(S_{W}^{-1}\right)^{*}Kf\\
&=&\sum_{i\in I}\sum_{j\in J_{i}}\left\langle \left(S_{W}^{-1}\right)^{*}Kf,\widetilde{f}_{ij}
\right
\rangle\pi_{R(K)} f_{ij}\\
&=&\sum_{i\in I,j\in J_{i}}\left\langle f,K^{*}S_{W}^{-1}\pi
_{S_{W}(R(K))}\omega_{i}\widetilde{f}_{ij}\right\rangle\pi_{R(K)}\omega_{i}f_{ij},
\end{eqnarray*}
for all $f\in \mathcal{H}$.
\end{proof}Let $W=\{(W_{i},\omega_{i})\}_{i\in I}$ be a $K$-fusion frame with the local Parseval frames
$\mathcal{F}_{i}=\{f_{ij}\}_{j\in J_{i}}$, for all $i\in I$.
 By using Theorem \ref{local} the sequence $\mathcal
{F}=\{\omega_{i}f_{ij}\}_{i\in I,j\in J_{i}}$ is a $K$-frame for  $\mathcal{H}$.
The canonical $K$-dual  $\{K^{*}S_{\mathcal{F}}^{-1}\pi_{S_{\mathcal{F}}(R(K))}\omega_{i}f_{ij}\}_{i\in I,j\in J
_{i}}$ could be performed canonical reconstruction of $R(K)$, i.e.
\begin{eqnarray*}
Kf=\sum_{i\in I,j\in J_{i}}\left\langle f,K^{*}S_{\mathcal{F}}^{-1}\pi
_{S_{\mathcal{F}}(R(K))}\omega_{i}f_{ij}\right\rangle\pi_{R(K)}\omega_{i}f_{ij},
\qquad(f\in \mathcal{H}).
\end{eqnarray*}
Then
 the $K$-duals  of $\mathcal{F}$ introduced above theorem and its canonical $K$-dual  are coincide.

In the following, we give a construction of $K$-fusion frames.

\begin{thm}
Let $\{(W_{i},\omega_{i})\}_{i\in I}$ be a fusion frame for $\mathcal{H}$  and $K\in B(\mathcal{H})$  a closed range operator such that  $W_{i}\subseteq R(K^{\dag})$, for all $i\in I$. Then $\{(KW_{i},\omega_{i})\}_{i\in I}$ is a $K$-fusion frame.
\end{thm}
\begin{proof}
Assume that $\{(W_{i},\omega_{i})\}_{i\in I}$ is a fusion frame with bounds $A$ and $B$, respectively. Due to Lemma \ref{v}  we obtain
\begin{eqnarray*}
A\|K^{*}f\|^{2}&\leq& \sum_{i\in I}\omega_{i}^{2}\left\|\pi_{W_{i}}K^{*}f\right\|^{2}\\
&\leq& \sum_{i\in I}\omega_{i}^{2}\left\|\pi_{W_{i}}K^{*}\pi_{KW_{i}}f\right\|^{2}\leq
\|K\|^{2}\sum_{i\in I}\omega_{i}^{2}\left\|\pi_{KW_{i}}\right\|^{2},
\end{eqnarray*}
for all $f\in \mathcal{H}$. In order to show the upper bound, using the hypothesis $W_{i}\subseteq R(K^{\dag})$ and Lemma \ref{v}  to $KW_{i}$ and $(K^{\dag})^{*}$ yields
\begin{eqnarray}\label{I}
\pi_{KW_{i}}(K^{\dag})^{*}=\pi_{KW_{i}}(K^{\dag})^{*}\pi_{W_{i}}.
\end{eqnarray}
As a consequence of (\ref{I}) we see that
\begin{eqnarray*}
\pi_{KW_{i}}=\pi_{KW_{i}}(K^{\dag})^{*}\pi_{W_{i}}K^{*},
\end{eqnarray*}
on $R(K)$. Also, for all $f\in \mathcal{H}$ there exist  $f_{1}\in R(K)$ and $f_{2}\in \left(R(K)\right)^{\bot}$ such that $f=f_{1}+f_{2}$. Therefore,
\begin{eqnarray*}
\sum_{i\in I}\omega_{i}^{2}\left\|\pi_{KW_{i}}f\right\|^{2}
&\leq&\sum_{i\in I}\left\|\pi_{KW_{i}}\left(f_{1}+f_{2}\right)\right\|^{2}\\
&\leq&\sum_{i\in I}\omega_{i}^{2}\left\|\pi_{KW_{i}}f_{1}\right\|^{2}\\
&\leq&\sum_{i\in I}\omega_{i}^{2}\left\|\pi_{KW_{i}}(K^{\dag})^{*}\pi_{W_{i}}K^{*}f_{1}\right\|^{2}\\
&\leq&B\|K^{\dag}\|^{2}\|K\|^{2}\|f_{1}\|^{2}\leq C\left\|f\right\|^{2},
\end{eqnarray*}
for some $C>0$.
\end{proof}
\begin{cor}
Let $T$ and $K$ be  bounded closed range operators on $\mathcal{H}$ and $\{(W_{i},\omega_{i})\}_{i\in I}$  a $K$-fusion frame on $\mathcal{H}$ such that $W_{i}\subseteq R((TK)^{\dag})$, for all $i\in I$. Then $\{(TW_{i},\omega_{i})\}_{i\in I}$ is a $TK$-fusion frame.
\end{cor}
\section{$K$-fusion frame multiplier}
The concept of multipliers was first introduced by Balazs  \cite{Basic} and developed by many authors \cite{Ra,rep,mitra}. In this section, we   generalize  multipliers associated to $K$-fusion frames and focus on the reconstruction of $R(K)$.
Throughout this section, we suppose $\{e_{i}\}_{i\in I}$ is an orthonormal basis for $\mathcal{H}$ and denote a family $\{(W_{i},\omega_{i})\}_{i\in I}$ by $W$.
\begin{defn}
Let $W$  be a $K$-fusion frame and   $m:=\{m_{i}\}_{i\in I}\in \ell^{\infty}$. For every Bessel fusion sequence $V,$ the operator
$\mathbf{M}_{m,W,V}:\mathcal{H}\rightarrow\mathcal{H}$  given by
\begin{eqnarray*}
\mathbf{M}_{m,W,V}=\sum_{i\in I}m_{i}\omega_{i}\upsilon_{i}\pi_{W_{i}}
(S_{W}^{-1})^{*}K\pi_{V_{i}}f,\qquad(f\in \mathcal{H}).
\end{eqnarray*}
is called a\textit{
$K$-fusion frame multiplier}.
\end{defn}
In the above definition, a $K$-fusion frame multiplier is a fusion frame multiplier if $K=I_{\mathcal{H}}$. For more details of fusion frame multipliers see \cite{mitra}.

If $W$ is a $K$-fusion frame, $V$ a Bessel fusion sequence in $\mathcal{H}$ and $m\in \ell^{\infty}$. Then
\begin{eqnarray}\label{T}
\mathbf{M}_{1,W,V}=T_{W}\psi_{wv}T_{V}^{*}.
 \end{eqnarray}
 In particular, $\mathbf{M}_{m,W,V}$  is bounded and
\begin{eqnarray*}
\|\mathbf{M}_{m,W,V}\|
\leq \sup_{i\in I}|m_{i}|\|S_{W}^{-1}\|\|K\|\sqrt{B_{W}B_{V}},
\end{eqnarray*}
where $B_{W}$ and $B_{V}$ are  upper bounds of $W$ and $V,$ respectively.

An operator $\mathcal{R}:\mathcal{H}\rightarrow\mathcal{H}$
 (resp. $\mathcal{L}:\mathcal{H}\rightarrow\mathcal{H}$) is  called a \textit{$K$-right inverse} (resp. \textit{$K$-left inverse}) of $\mathbf{M}_{m,W,V}$ if
\begin{eqnarray*}
\mathbf{M}_{m,W,V}\mathcal{R}=K\qquad(resp.  \quad \mathcal{L}\mathbf{M}_{m,W,V}=K).
\end{eqnarray*}
Also, by a $K$-inverse we mean both a $K$-right inverse and a $K$-left inverse.

In the following theorem we show that every $K$-dual of a $K$-fusion frame $W$ is a $K^{*}$-fusion frame.
\begin{thm}
Let $W$ be a $K$-fusion frame and $V$ a Bessel fusion sequence. The following assertions hold.
\begin{enumerate}
\item\label{Re1} Let $\mathbf{M}_{m,W,V}=K$. Then  $V$ satisfies the lower $K^{*}$-fusion frame condition. In particular, it is  a $K^{*}$-fusion frame.
    \item\label{Re2} If $\mathbf{M}_{m,W,V}$ has a $K$-left inverse, then $V$ is a $K^{*}$-fusion frame. 
\end{enumerate}
\end{thm}
\begin{proof}
\ref{Re1}. Using Cauchy-Schwartz inequality,   we obtain
\begin{eqnarray*}
\|Kf\|^{2}&=&\left|\left\langle \mathbf{M}_{m,W,V}f,Kf\right\rangle\right|\\
&=&\left|\left\langle \sum_{i\in I}m_{i}\omega_{i}\upsilon_{i}\pi_{W_{i}}
  (S_{W}^{-1})^{*}K\pi_{V_{i}}f,Kf\right\rangle\right|\\
  &\leq &\sum_{i\in I}|m_{i}|\omega_{i}\upsilon_{i}\left\langle (S_{W}^{-1})^{*}K\pi_{V_{i}}f,\pi_{W_{i}}Kf\right\rangle\\
  &\leq &\sum_{i\in I}|m_{i}|\omega_{i}\upsilon_{i}\left\|(S_{W}^{-1})^{*}K\pi_{V_{i}}f\right\|
  \left\|\pi_{W_{i}}Kf\right\|\\
  &\leq& \sup_{i\in I}|m_{i}| \left(\sum_{i\in I}\upsilon_{i}^{2}\left\|(S_{W}^{-1})^{*}K\pi_{V_{i}}f\right\|^{2}\right)^
  \frac{1}{2}\left(\sum_{i\in I}\omega_{i}^{2}\left\|\pi_{W_{i}}Kf\right\|^{2}\right)^
  \frac{1}{2}\\
  &\leq&\sup_{i\in I}|m_{i}|\|(S_{W}^{-1})^{*}K\|\sqrt{B_{W}}\|Kf\|\left(\sum_{i\in I}\upsilon_{i}^{2}\left\|\pi_{V_{i}}f\right\|^{2}\right)^
  \frac{1}{2},
\end{eqnarray*}
for all $f\in \mathcal{H}$, where $B_{W}$ is an upper bound of $W$. Hence $V$ is a $K^{*}$- fusion frame.

\ref{Re2}.  Let $\mathcal{L}$ be a $K$-left inverse of $\mathbf{M}_{m,W,V}$. Applying (\ref{T}) for every  $g\in \mathcal{H}$  yields
\begin{eqnarray*}
\|Kg\|^{2}&=&\left|\left\langle Kg,Kg\right\rangle\right|\\
&=&\left|\left\langle Kg,\mathcal{L}\mathbf{M}_{m,W,V}g\right\rangle\right|\\
&=&\left|\left\langle \mathcal{L}^{*}Kg,\mathbf{M}_{m,W,V}g\right\rangle\right|\\
&\leq& \sup_{i\in I}|m_{i}|\left\|\mathcal{L}^{*}\right\|\|Kg\|\left\|T_{W}\psi_{wv}T_{V}^{*}g\right\|\\
&\leq &\sup_{i\in I}|m_{i}|\left\|\mathcal{L}\right\|\|Kg\| \left\|S_{W}^{-1}\right\|\|K\|\sqrt{B_{W}}\left(\sum_{i\in I}\upsilon_{i}^{2}\left\|\pi_{V_{i}}g\right\|^{2}\right)^
  \frac{1}{2}.
\end{eqnarray*}
This completes the proof.
\end{proof}
The invertibility of $K$-fusion frame multipliers is studied in the following.

\begin{thm}
Let $W=\{(W_{i},1)\}_{i\in I}$ and $V=\{(V_{i},1)\}_{i\in I}$ be  two $K$-fusion frames such that
\begin{eqnarray}\label{inv2}
\sum_{i\in I}\left\|\pi_{R(K)}\left(S_{V}^{-1}\right)^{*}K\pi_{W_{i}}-
\pi_{V_{i}}\right\|^{2}
<\frac{A_{V}^{2}}{B_{V}
\left\|K^{\dag}\right\|^{4}},\qquad(f\in R(K)),
\end{eqnarray}
where $A_{V}$ and $B_{V}$ are the lower and upper bounds of $V,$ respectively. Then $\mathbf{M}_{1,V,W}$ is invertible on $R(K)$.
\end{thm}
\begin{proof}
By using (\ref{karan}) the operator $S_{V}:R(K)\rightarrow S_V(R(K))$ is invertible and $\left\|S_{V}^{-1}\right\|\leq\frac{\left\|K^{\dag}\right\|^{2}}{A_{V}}$. So,
\begin{eqnarray*}
 \left\|\mathbf{M}_{1,V,W}-S_{V}\right\|^{2}&=&\left\|\sum_{i\in I}
 \pi_{V_{i}}\left(S_{V}^{-1}\right)^{*}K\pi_{W_{i}}-\sum_{i\in I}\pi_{V_{i}}\right\|^{2}\\
 &=& \left\|\sum_{i\in I}\pi_{V_{i}}
 \left(\left(S_{V}^{-1}\right)^{*}K\pi_{W_{i}}-\pi_{V_{i}}\right)
 \right\|^{2}\\
 &\leq&\sum_{i\in I}B_{V}\left\|\pi_{R(K)}\left(S_{V}^{-1}\right)^{*}K\pi_{W_{i}}-\pi_{V_{i}}\right\|^{2}
 \\
 &\leq& \frac{A_{V}^{2}}{\left\|K^{\dag}\right\|^{4}}\leq \frac{1}{\|S_{V}^{-1}\|^{2}}.
 \end{eqnarray*}
Moreover, $\mathbf{M}_{1,V,W}$ maps $R(K)$ into $S_{V}(R(K))$. Therefore,
 \begin{eqnarray*}
 \left\|I_{S_{V}(R(K))}-\mathbf{M}_{1,V,W}S_{V}^{-1}\right\|^{2}&=&
 \left\|\left(
 S_{V}-\mathbf{M}_{1,V,W}\right)S_{V}^{-1}\right\|^{2}\\
 &\leq&\left\|
 S_{V}-\mathbf{M}_{1,V,W}\right\|^{2}\left\|S_{V}^{-1}\right\|^{2}\\
 &<& \frac{1}{\|S_{V}^{-1}\|^{2}}\left\|S_{V}^{-1}\right\|^{2}=1.
 \end{eqnarray*}
 Hence, $\mathbf{M}_{1,W,V}S_{V}^{-1}$ is invertible on $S_{V}(R(K))$, by Theorem 8.1 of \cite{book} and so  $\mathbf{M}_{1,V,W}$ is invertible on $R(K)$.
\end{proof}
The composition of frame multipliers \cite{Basic} and  fusion frame multipliers \cite{mitra} were investigated.
In the following, we discuss about the composition of $K$-fusion frame multipliers.
\begin{thm}
Let  $W=\{(W_{i},1)\}_{i\in I}$ be a $K$-fusion frame and $Z=\{(Z_{i},1)\}_{i\in I}$ a $L$-fusion frame. Also, let $V=\{(V_{i},1)\}_{i\in I}$ and $X=\{(X_{i},1)\}_{i\in I}$ be biorthogonal  Bessel fusion sequences. Then
\begin{eqnarray}\label{com}
\mathbf{M}_{1,W,V}\mathbf{M}_{1,Z,X}=\mathbb{M}_{1,
\left\{\pi_{W_{i}}(S_{W}^{-1})^{*
}K\pi_{V_{i}}\pi_{Z_{i}}e_{j}\right\}_{i,j\in I},\left\{\pi_{X_{i}}L^{*}
S_{Z}^{-1}\pi
_{S_{Z}(R(L)}e_{j}\right\}_{i,j\in I}}
\end{eqnarray}
\end{thm}
\begin{proof}
Using the biorthogonality of $Z$ and $V$, it follows that
\begin{eqnarray*}
&&\mathbf{M}_{1,W,V}\mathbf{M}_{1,Z,X}f\\&=&\mathbf{M}_{1,W,V}
\left(\sum_{i\in I}\pi_{Z_{i}}\left(S_{Z}^{-1}\right)^{*}L\pi_{X_{i}}f\right)\\
&=& \mathbf{M}_{1,W,V}\left(\sum_{i,j\in I}\left\langle \left(S_{Z}^{-1}\right)^{*}L\pi_{X_{i}}f,e_{j}\right\rangle
\pi_{Z_{i}}e_{j}\right)\\
&=&\mathbf{M}_{1,W,V}\left(\sum_{i,j\in I}\left\langle f,\pi_{X_{i}}L^{*}S_{Z}^{-1}\pi_{S_{Z}(R(L))}e_{j}\right\rangle
\pi_{Z_{i}}e_{j}\right)\\
&=&\sum_{m,n,i,j\in I}\left\langle f,\pi_{X_{i}}L^{*}S_{Z}^{-1}\pi_{S_{Z}(R(L))}e_{j}\right\rangle
\left\langle
\pi_{Z_{i}}e_{j},\pi_{V_{m}}K^{*}S_{W}^{-1}\pi_{S_{W}(R(K))}
e_{n}\right\rangle
\pi_{W_{m}}e_{n}\\
&=&\sum_{i,j,n\in I}
\left\langle f,\pi_{X_{i}}L^{*}S_{Z}^{-1}\pi_{S_{Z}(R(L))}e_{j}\right\rangle
\left\langle
\left(S_{W}^{-1}\right)^{*}K\pi_{V_{i}}\pi_{Z_{i}}e_{j},e_{n}
\right\rangle
\pi_{W_{i}}e_{n}\\
&=&\sum_{i,j\in I}\left\langle f,\pi_{X_{i}}L^{*}S_{Z}^{-1}\pi_{S_{Z}(R(L))}e_{j}\right
\rangle\pi_{W_{i}}
\left(S_{W}^{-1}\right)^{*}K
\pi_{V_{i}}\pi_{Z_{i}}e_{j}\\
&=&\mathbb{M}_{1,\left\{\pi_{W_{i}}(S_{W}^{-1})^{*
}K\pi_{V_{i}}\pi_{Z_{i}}e_{j}\right\}_{i,j\in I},\left\{\pi_{X_{i}}L^{*}
S_{Z}^{-1}\pi
_{S_{Z}(R(L))}e_{j}\right\}_{i,j\in I}}f,
\end{eqnarray*}
for all $f\in \mathcal{H}$.
\end{proof}

Note that, if $W=\{(W_{i},1)\}_{i\in I}$ is a $K$-fusion frame, $V=\{(V_{i},1)\}_{i\in I}$  an orthonormal fusion basis  and $H=\{(H_{i},1)\}_{i\in I}$ a Bessel fusion sequence for $\mathcal{H}$ such that $H_{i}\subseteq V_{i}$, for all $i\in I$. Then
\begin{eqnarray*}
\mathbf{M}_{1,W,V}\mathbf{M}_{1,V,H}=\mathbf{M}_{1,W,H}.
\end{eqnarray*}

\bibliographystyle{amsplain}

\end{document}